\documentclass[11pt]{amsart}
\usepackage{graphicx}
\usepackage{mathpazo}
\usepackage{latexsym}   
\usepackage{amssymb}    
\usepackage{amsmath}    
\usepackage{amsthm}     
\usepackage{hyperref}
\usepackage{xspace}

\usepackage[margin=1in]{geometry}

\numberwithin{equation}{section}
\newtheorem{theorem}{Theorem}[section]

\newtheorem{lemma}[theorem]{Lemma}

\newtheorem{proposition}[theorem]{Proposition}
\newtheorem{corollary}[theorem]{Corollary}

\theoremstyle{definition}

\usepackage{diagrams}
\newarrow{to}---->
\newarrow{dashedto}{}{dash}{}{dash}>
\newarrow{mto}|--->
\newarrow{impto}===={=>}
\newarrowtail{<=}\Rightarrow\Leftarrow\Uparrow\Downarrow
\newarrow{bboth}{<=}==={=>}
\newarrow{inject}{hooka}---{vee}
\newarrow{surject}----{>>}

\def\et{\text{\'et}}
\newcommand{\stk}[1]{{\mathcal #1}}
\def\btn{${\rm BT}_n$\xspace}
\def\btinfinity{${\rm BT}_\infty$\xspace}
\def\btone{${\rm BT}_1$\xspace}

\DeclareMathOperator{\twist}{Twist}
\DeclareMathOperator{\Isom}{Isom}
\DeclareMathOperator{\cent}{Z}
\def\aalpha{{\pmb\alpha}}
\def\dieu{{\mathbb D}}
\def\setcomp{\smallsetminus}
\def\bigoh{{\rm O}}
\renewcommand{\bar}[1]{{\overline{#1}}}

\def\ff{{\mathbb F}}
\def\integ{{\mathbb Z}}
\def\gp{{\mathbb G}}
\def\iso{\cong}
\def\ra{\rightarrow}
\def\sriso{\stackrel{\sim}{\rightarrow}}
\newcommand{\st}[1]{\left\{#1\right\}}
\def\mmu{{\pmb\mu}}
\newcommand{\oneover}[1]{\frac{1}{#1}}
\DeclareMathOperator{\im}{im}
\def\comp{\circ}
\newcommand{\dirlim}[1]{\lim_{\stackrel{\rightarrow}{#1}}}
\newcommand{\til}[1]{{\widetilde{#1}}}
\DeclareMathOperator{\aut}{Aut}
\DeclareMathOperator{\gal}{Gal}
\DeclareMathOperator{\defo}{Def}
\def\cross{\times}
\def\calg{{\mathcal G}}
\def\rat{{\mathbb Q}}
\newcommand{\rest}[1]{|_{#1}}
\newcommand{\abs}[1]{{\left|#1\right|}}
\def\inject{\hookrightarrow}
\def\cala{{\mathcal A}}
\def\calo{{\mathcal O}}
\DeclareMathOperator{\End}{End}
\def\units{^\cross}
\DeclareMathOperator{\mat}{Mat}
\DeclareMathOperator{\spec}{Spec}
\DeclareMathOperator{\gl}{GL}
\DeclareMathOperator{\jac}{Jac}
\def\cali{{\mathcal I}}
\def\inv{^{-1}}
\DeclareMathOperator{\ord}{ord}
\def\twiddle{\sim}
\DeclareMathOperator{\frob}{Fr}
\global\let\hom\undefined
\DeclareMathOperator{\hom}{Hom}

\newcommand{\ang}[1]{{{\langle #1 \rangle}}}

\newenvironment{alphabetize}{\begin{enumerate}

}{\end{enumerate}}

\begin{document}

\title{The distribution of torsion subschemes of abelian varieties}

\author{Jeffrey D. Achter}
\email{j.achter@colostate.edu}
\address{Department of Mathematics, Colorado State University, Fort
Collins, CO 80523} 
\urladdr{http://www.math.colostate.edu/~achter}
\thanks{This work was partially supported by grants from the Simons
  Foundation (204164) and the NSA (H98230-14-1-0161).}
  
  \begin{abstract}
  We consider the distribution of $p$-power group schemes among the torsion of abelian varieties over finite fields of characteristic $p$, as follows.  Fix natural numbers $g$ and $n$, and let $\xi$ be a non-supersingular principally quasipolarized Barsotti-Tate group of level $n$.  We classify the $\ff_q$-rational forms $\xi^\alpha$ of $\xi$.  Among all principally polarized abelian varieties $X/\ff_q$ of dimension $g$ with $X[p^n]_{\bar\ff_q} \iso \xi_{\bar\ff_q}$, we compute the frequency with which $X[p^n] \iso \xi^\alpha$.  The error in our estimate is bounded by $D/\sqrt q$, where $D$ depends on $g$, $n$ and $p$, but not on $\xi$.
  \end{abstract}

\maketitle

\section{Introduction}

Either the general philosophy of arithmetic statistics or the specific
question of Cohen-Lenstra heuristics for function fields can easily
lead one to wonder:
\begin{quotation}
How are group schemes distributed among the $N$-torsion group schemes
of principally polarized abelian varieties over a finite field?
\end{quotation}
More concretely, for various combinations of $g$, $q = p^e$, and $N$,
we would like to understand the multiset
\[
\st{ (X[N],\lambda_N) : (X,\lambda) \in \stk A_g(\ff_q) }.
\]
For example, suppose $N$ is relatively prime to $p$ and $g$ is fixed.
The $N$-torsion of the universal abelian scheme $\stk \stk \stk X[N] \ra \stk
A_g$ is a local system of $(\integ/N)$-modules, and the scheme of isomorphisms
$\Isom((\stk X[N],\lambda_N), ((\integ/N)^{\oplus
  2g},\ang{\cdot,\cdot}))$ is an \'etale, Galois cover of $\stk A_g$.
A monodromy calculation shows that this cover is irreducible; the
Chebotarev theorem then gives good control, as $q\ra\infty$, of the
distribution of forms of $(\integ/N)^{\oplus 2g}$.  (See, e.g.,  \cite{achtercag}.)


The landscape looks dramatically different if, instead of considering
prime-to-$p$ torsion, one analyzes $p^n$-torsion.  For example,
there are several different geometric isomorphism classes of group
scheme which arise as $X[p]$.  Perhaps worse yet, there are
  infinitely many different geometric isomorphism classes of group
schemes which arise as $X[p^2]$ over $\bar\ff_p$.  Therefore, a direct
strategy based on local systems and the Chebotarev theorem seems
difficult to implement.

Still, not all is lost.  In \cite{caisetal13}, the authors
fix a base field $\ff_q$ and describe a natural measure on the space
of principally quasipolarized group schemes of rank $2g$ which are
killed by $p$.  They then are able to compute the large-$g$
expectation of various events, such as having specified $a$-number or
$p$-rank.  Unfortunately, it is unknown whether this
measure accurately describes the distribution of the $p$-torsion group
schemes of abelian varieties, although numerical experiments with hyperelliptic Jacobians are encouraging.

In the present paper, we take a complementary, ``fixed $g$, large
$q$'' approach.  An initial, na\"ive attempt fails to detect many of
the interesting structures on $\stk A_g$ in positive characteristic.
The ordinary locus is open and dense, and so in the large $q$ limit,
(essentially) all $X[p]$ are forms of $(\mmu_p \oplus \integ/p)^g$.
To detect more subtle variations, we {\em condition} on geometric
isomorphism class, as follows.

Let $\xi$ be a geometric isomorphism class of principally
quasipolarized \btn; let $\stk A_{g,\xi}$ be the locus where
$\stk X[p^n]$ is geometrically isomorphic to $\xi$.  We will show that
there is a finite group $A(\xi)$ whose conjugacy classes
$A(\xi)^\sharp$ parametrize, for  suitable
fields $\ff_q$, $\ff_q$-rational forms of $\xi$. 
We essentially show (see Theorem \ref{thmainacirc} for details) that, for each $\alpha
\in A(\xi)^\sharp$, the proportion of elements of $(X,\lambda) \in \stk
A_{g,\xi}(\ff_q)$ for which $(X,\lambda)[p^n] \iso \xi^\alpha$
approaches $1/\#Z_{A(\xi)}(\alpha)$ in a surprisingly uniform way.  In
fact, 
\begin{equation}
\label{eqmainintro}
\frac{\#\st{ (X,\lambda) \in \stk A_{g,\xi}(\ff_q) : (X,\lambda)[p^n] \iso
  \xi^\alpha}}{\# \stk A_{g,\xi}(\ff_q)} = \oneover{\#Z_{A(\xi)}(\alpha)}
+\bigoh_{g,n,p}(1/\sqrt q);
\end{equation}
the error is bounded by $D/\sqrt q$ for some constant $D$ which
depends only on $g$, $n$ and $p$.  Thanks to Lemma \ref{lemauttwist}, this is compatible with the broad philosophy of \cite{caisetal13}, and even that of \cite{cohenlenstra}; a (polarized) group (scheme) occurs with frequency inversely proportional to the size of its automorphism group.

Our analysis starts (Section \ref{secpointwise}) with a classification of
the possible $\ff_q$-rational forms of a given principally
quasipolarized \btn $\xi$.  One fact which emerges is that there is a
finite set which controls such forms, independent of the base field.
  In this way, it is possible to compare forms of $\xi$ over
ever-larger finite extensions of the base field.

In Section \ref{secstratifications} we review (and, where necessary,
extend) results, largely due to Oort and to Chai and Oort, on the
structure of the strata $\stk A_{g,\xi}$.  It turns out that, although
there are infinitely many such strata (if $n \ge 2$), their topology
(if $n\gg_g 0$) may be uniformly bounded in terms of $g$, $n$ and $p$.

We apply these results in Section \ref{secdistribution} to state and
prove our main equidistribution result.  As is often the case, a geometric
Chebotarev theorem is the crucial guarantor of equidistribution, but
it is invoked here in a somewhat novel way.    If $n$ is small
relative to $g$, then there is often no single local system on $\stk
A_{g,\xi}$ which controls the variation in the forms of $\xi$.  Still,
we express $\stk A_{g,\xi}$ as a disjoint union of infinitely many
central leaves $\stk A_{g,\til\xi}$.  (Of course, only finitely many
of them are defined over a given finite field.)  Since each of {\em these} has
bounded topology,  we are able to deduce the uniform bound asserted in \eqref{eqmainintro}.

  The paper concludes (Section
\ref{secexample}) with a series of examples, and a brief discussion of
related results.

The present investigation was inspired by the work of Cais, Ellenberg
and Zureick-Brown, even though the results here are incomparable with
those of \cite{caisetal13}.  The author acknowledges helpful
discussions with Oort and Liedtke, especially concerning the example in  Section
\ref{subsecexss}.  In particular, the explicit classification of
Dieudonn\'e modules there was communicated to
the author by Oort.  Finally, this paper benefited from the referee's careful reading.

\section{Forms of a \texorpdfstring{\btn}{BTn}}
\label{secpointwise}

\subsection{Notation}

We will work over finite fields of characteristic $p>0$; in
particular, $q$ is always a power of $p$.  Let $\ff$ be an algebraic
closure of $\ff_p$.

Let $K$ be a perfect field of characteristic $p$, and let $n$ be a
natural number.  A Barsotti-Tate group of level $n$, or  \btn, over $K$
is a commutative finite group scheme $G/K$ which is annihilated by
$p^n$ and flat over $\integ/p^n$.  (If $n= 1$, one further insists
that $\im F: G \ra G^{(p)} = \ker V: G^{(p)} \ra G$.)  A
Barsotti-Tate, or $p$-divisible, group is an inductive limit $G =
\dirlim{n} G_n$ where $G_n$ is a \btn and $[p]$ is an epimorphism; we
will sometimes refer to such an object as a \btinfinity.

Let $G$ be a \btn.  The Cartier dual of $G$ is $G^D = \hom(G,\gp_m)$.
If $p> 2$, a principal quasipolarization on $G$ is an isomorphism
$\lambda:G \sriso G^D$ such that $\lambda^D = -\lambda\comp \kappa$,
where $\kappa: G^{DD} \sriso G$ is the canonical isomorphism.  (In
characteristic $2$, this definition must be modified slightly; see
\cite[2.6]{moonengsas} or \cite[6.3]{gabbervasiu} for more details.) The pair
$(G,\lambda)$ is called a principally quasipolarized, or pqp, group
scheme (or \btn).

Define:
\begin{itemize}
\item $\Xi_{g,n}(\ff)$ for $1 \le n \le \infty$ is the set of
  $\ff$-isomorphism classes of principally quasipolarized \btn's of
  height $2g$ over $\ff$.

\item For $\xi \in \Xi_{g,n}(\ff)$, let $\xi(\ff_q)$ be the set of
  $\ff_q$-isomorphism classes of pqp group schemes $(G,\lambda)/\ff_q$
  such that the class of $(G,\lambda)_{\bar\ff_q}$ in $\Xi(\ff)$ is $\xi$.
\end{itemize}
For each of these sets  we will write, e.g., $(G,\lambda) \in
\xi(\ff_q)$ (or even $G \in \xi(\ff_q)$) when
we really mean a representative for the $\ff_q$-isomorphism class.  For $1 \le n \le m \le \infty$ there is an obvious truncation functor $\Xi_{g,m}(\ff) \ra \Xi_{g,n}(\ff)$, which we will denote by $\til \xi \mapsto \til \xi[p^n]$.  Let $\Xi_{g,m}(\ff)_\xi$ be the fiber of this map over $\xi$.

For an abstract group $\Gamma$, let $\Gamma^\sharp$ be the set of
conjugacy classes of $\Gamma$, and let $\cent(\Gamma)$ be the center
of $\Gamma$.

\subsection{Automorphism groups and forms}

Let $K$ be a finite field of characteristic $p$, let $n$ be a positive integer, and let
$(G,\lambda)/K$ be a pqp-\btn.  Its automorphism group scheme $\aut_{G,\lambda}$ is affine,
 and sits in the usual \'etale-connected sequence
\begin{diagram}[LaTeXeqno]
\label{diagconnetale}
0 & \rto & \aut_{G,\lambda}^0 & \rto  & \aut_{G,\lambda} & \rto & \pi_0(\aut_{G,\lambda}) & \rto & 0.
\end{diagram}
Then $\pi_0(\aut_{G,\lambda})(\bar K)$ is a $\gal(K)$-module.  We will find it
convenient to call $(G,\lambda)$ split if the group $\pi_0(\aut_{G,\lambda})(\bar K)$ has trivial
$\gal(K)$-structure, i.e., if $\pi_0(\aut_{G,\lambda})(K) = \pi_0(\aut_{G,\lambda})(\bar 
K)$.  For a split pqp group scheme  $(G,\lambda)/K$, let
$A_{G,\lambda}$ be the abstract, finite group
$\pi_0(\aut_{G,\lambda})(K)$.  

Since $K$ is perfect, \eqref{diagconnetale} splits canonically, and we
sometimes view $A_{G,\lambda}$ as a (finite, discrete) group of
automorphisms of $(G,\lambda)$.

For a finite extension $L/K$, let $\twist(L/K,(G,\lambda))$ be the set of
$K$-isomorphism classes of pqp group schemes $(H,\mu)/K$ such that $(H,\mu)_L\iso
(G,\lambda)_L$. On one hand, $\twist(L/K,(G,\lambda))$ is a pointed set, with distinguished
element the $K$-isomorphism class of $(G,\lambda)$ itself.  On the
other hand, we endow $A_{G,\lambda}^\sharp$ with the structure of a
pointed set by distinguishing (the conjugacy class of) the identity
element. Let $\twist(K,(G,\lambda)) =
\cup_{L/K} \twist(L/K,(G,\lambda))$.

\begin{lemma}
\label{lemclassifytwists}
Let $(G,\lambda)/K$ be a split pqp group scheme over a finite field.  There is a bijection of pointed sets
between
\begin{itemize}
\item $\twist(K,(G,\lambda))$; and
\item $A_{G,\lambda}^\sharp$.
\end{itemize}
\end{lemma}

\begin{proof}
It suffices to prove that for each finite extension $L/K$, there is a
bijection of pointed sets between $\twist(L/K,(G,\lambda))$ and conjugacy classes of elements of
$A_{G,\lambda}$ of order dividing $[L:K]$.
By \cite[III.1.3]{serregalcoh}, there is a bijection between
$\twist(L/K,(G,\lambda))$ and $H^1(L/K,\aut_{G,\lambda}(L))$.  Consider the exact
sequence of groups
\begin{diagram}
0 & \rto & \aut_{G,\lambda}^0(L) & \rto & \aut_{G,\lambda}(L) & \rto & \pi_0(\aut_{G,\lambda})(L)
& \rto &0
\end{diagram}
and the corresponding long exact sequence of Galois cohomology groups
\begin{diagram}
H^1(L/K, \aut_{G,\lambda}^0(L)) & \rto & H^1(L/K, \aut_{G,\lambda}(L)) & \rto
& H^1(L/K, \pi_0(\aut_{G,\lambda})(L)) & \rto & H^2(L/K, \aut_{G,\lambda}^0(L)).
\end{diagram}
The left-most term is zero by Lang's theorem; the rightmost term
is trivial, since $K$ is a finite field and thus has cohomological
dimension $\le 1$.  Since $G$ is split, $\pi_0(\aut_{G,\lambda})(L)\iso A_{G,\lambda}$,
and thus $H^1(L/K,\aut_{G,\lambda}(L)) \iso H^1(L/K, A_{G,\lambda})$.  Moreover,
$\gal(L/K)$ is cyclic and $A_{G,\lambda}$ is trivial as $\gal(L/K)$-module,
whence the asserted description of $H^1(L/K,A_{G,\lambda})$.
\end{proof}

In particular, (the isomorphism class of) a twist $(H,\mu)\in \twist(K,(G,\lambda))$
is represented by some $\alpha \in A_{G,\lambda}$, well-defined up to conjugacy.

\begin{lemma}
\label{lemauttwist}
Suppose $H \in \twist(K,(G,\lambda))$ is represented by $\alpha \in
A_{G,\lambda}$.  
\begin{alphabetize}
\item We have
\[
\pi_0(\aut_{H,\mu})(K) \iso \cent_{A_{G,\lambda}}(\alpha).
\]
\item If
$L$ is a finite extension of $K$, then $H_L \in \twist(L,(G,\lambda)_L)$ is
represented by $\alpha^{[L:K]} \in A_{(G,\lambda)_L} \iso
A_{G,\lambda}$.
\end{alphabetize}
\end{lemma}

\begin{proof}
Part (a) is a special case of a more general statement worked out in
\cite[I.5.3.34]{serregalcoh}; part (b) follows immediately from the fact
that the $[L:K]^{{\rm th}}$ power of a topological generator of
$\gal(K)$ is a topological generator of $\gal(L)$.
\end{proof}

In particular, in the context of Lemma \ref{lemauttwist}, $(H,\mu)$ is
split if and only if $\alpha \in \cent(A_{G,\lambda})$.  We will call such an
$(H,\mu)$ a central twist of $(G,\lambda)$.

\begin{lemma}
\label{lemtwistrs}
Let $\xi \in \Xi_{g,n}(\ff)$.  Suppose $(G,\lambda) \in \xi(\ff_{q^r})$ and
$(H,\mu)\in \xi(\ff_{q^s})$ are split and that $\gcd(r,s) = 1$.  Then there
are split $(G',\lambda') \in \xi(\ff_{q^r})$ and $(H',\mu') \in \xi(\ff_{q^s})$ such that 
$(G',\lambda')_{\ff_{q^{rs}}} \iso
(H',\mu')_{\ff_{q^{rs}}}$.
\end{lemma}

\begin{proof}
Since $(G,\lambda)_{\ff_{q^{rs}}}$ and $(H,\mu)_{\ff_{q^{rs}}}$ are both split,
$(H,\mu)_{\ff_{q^{rs}}}$ is a central twist of $(G,\lambda)_{\ff_{q^{rs}}}$; there is
some $a \in \cent(A_{G,\lambda}) \subset A_{G,\lambda}$ such that twisting
$(G,\lambda)_{\ff_{q^{rs}}}$ by the cocycle $\sigma \mapsto a$ yields
a pqp group
scheme isomorphic to $(H,\mu)_{\ff_{q^{rs}}}$.  A split twist $(G',\mu')$ of 
$(G,\mu)_{\ff_{q^s}}$ corresponds to some $b \in \cent(A_{G,\lambda})$.  Using (in
this lemma only) additive notation for $\cent(A_{G,\lambda})$, we find that
$(G',\mu')_{\ff_{q^{rs}}}$ corresponds to $sb$ (Lemma \ref{lemauttwist}(b)).   Similarly, if $(H',\mu')$ is a split twist
of $(H,\mu)$ corresponding to $c \in \cent(A_{H,\mu})$, then $(H',\mu')_{\ff_{q^{rs}}}$ is
represented by $sc$ as a twist of $(H,\mu)_{\ff_{q^{rs}}}$, and by $rc+a$ as
a twist of $(G,\lambda)_{\ff_{q^{rs}}}$ \cite[I.5.3.35]{serregalcoh}.  Consequently, in order to prove the
lemma, we must show that for any $a \in \cent(A_{G,\lambda})$, there exist $b,c
\in \cent(A_{G,\lambda})$ such that
\[
rc+a = sb.
\]
If we write $\cent(A_{G,\lambda}) \iso \oplus \integ/m_i$ and recall that $r$
and $s$ are relatively prime, the existence of $b$ and $c$ is obvious.
\end{proof}

\begin{proposition}
\label{propminfield}
Let $\xi \in \Xi_{g,n}(\ff)$.  There exists a finite field
$\ff_{q(\xi)}$ and a split pqp group scheme $(G,\lambda)\in \xi(\ff_{q(\xi)})$ such
that, if $(H,\mu) \in \xi(\ff_q)$ is split, then $\ff_q$ contains
$\ff_{q(\xi)}$ and $(H,\mu)$ is a central twist of $(G,\lambda)_{\ff_q}$.
\end{proposition}

\begin{proof}
Given Lemma \ref{lemtwistrs}, it suffices to prove that if $\gcd(r,s) = 1$, and
if there are split pqp group schemes $(G_r,\lambda_r) \in
\xi(\ff_{q_r})$ and $(G_s,\lambda_s) \in \xi(\ff_{q^s})$ which are
isomorphic to the pqp group scheme $(G,\lambda) \in
\xi(\ff_{q^{rs}})$, then $(G,\lambda)$ descends to $\ff_q$.  Let
$\sigma$ generate $\gal(\ff_{q^{rs}}/\ff_q)$; by hypothesis,
$(G,\lambda)$ is canonically isomorphic to its pullbacks by $\sigma^s$
and $\sigma^r$, and thus is invariant under
$\gal(\ff_{q^{rs}}/\ff_q)$.  Descent is effective for affine group
schemes, and the proposition is proved.
\end{proof}

Set
\[
A(\xi)= A_{G,\lambda},
\]
where $(G,\lambda)\in \xi(\ff_q)$ is split.  As above, we note that
$A(\xi)$ acts on $(G,\lambda)$.

If $\alpha \in A(\xi)$ and if $\ff_q$ contains $\ff_{q(\xi)}$, let
\[
\xi^{\alpha,\ff_q} \in \xi(\ff_q)
\]
denote the form of $\xi$ corresponding to the conjugacy class of
$\alpha$ in $A(\xi)$.  

This notation is slightly ambiguous; the
isomorphism class of $\xi^{\alpha,\ff_q}$ actually depends on the
choice of split model $(G,\lambda)$ for $\xi$.  Still, suppose that
$(G_1,\lambda_1)$ and $(G_2,\lambda_2)$ are split models for $\xi$.
Recall (as seen in
Lemma \ref{lemauttwist}) that, after a choice of isomorphism
$A_{G_1,\lambda_1}\ra A_{G_2,\lambda_2}$ of abstract groups, the
bijection between $\twist(\ff_q, (G_1,\lambda_1))$ and
$\twist(\ff_q,(G_2,\lambda_2))$ is realized by
\begin{diagram}
A_{G_1,\lambda_1}^\sharp & \rto & A_{G_2,\lambda_2}^\sharp \\
[\alpha] & \rmto & [\gamma\alpha]
\end{diagram}
where $\gamma$ is some element of the center.  Now suppose that
$(H,\mu)$ is an $\ff_q$-form of $\xi$, represented by $[\alpha_i]$ in
$\twist(\ff_q(G_i,\lambda_i))$.  Since the centralizers of $\alpha_1$
and $\gamma\alpha_1$ coincide, the quantity
\[
\# Z_{A_{(G_i,\lambda_i)}}(\alpha_i)
\]
is independent of the choice of split model for $\xi$.

\subsection{Lifts}

Suppose $1 \le n \le m \le \infty$ and
 $\til \xi \in \Xi_{g,m}(\ff)$; let $\xi = \til\xi[p^n]\in
\Xi_{g,n}(\ff)$ be its $n$-truncation.  We define a subgroup $A(\til
\xi, \xi)\subseteq A(\xi)$ as follows.

Choose a split representative $(\til G, \til\lambda)/\ff_q$ for $\til
\xi$, and let $(G,\lambda) = (\til G[p^n], \til\lambda[p^n])$.  The
natural morphism of functors induces a map of group schemes
\begin{diagram}
\aut_{\til G,\til\lambda} & \rto^{\rho} & \aut_{G,\lambda}
\end{diagram}
and in particular a map of \'etale quotients
\begin{diagram}
\pi_0(\aut_{\til G,\til\lambda}) & \rto^{\pi_0(\rho)}&
\pi_0(\aut_{G,\lambda}).
\end{diagram}
Let
\[
A(\til \xi, \xi) = \im(\pi_0(\rho)).
\]
As an abstract group, it is independent of the choice of model $(\til
G, \til\lambda)$, and we view it as the group of discrete automorphisms of $\xi$ which lift to $\til \xi$.

\section{Stratifications on \texorpdfstring{$\stk A_g$}{Ag}}
\label{secstratifications}

Let $N\ge 3$ be an integer relatively prime to $p$.   Let $\stk A_{g,d,N}$ be the moduli space, over $\ff_p$, of $g$-dimensional abelian varieties equipped with a degree $d^2$ polarization and symplectic level-$N$ structure, and let
\[
\stk A_g = \stk A_{g,1,N}.
\]
Let $\stk X \ra \stk A_g$
be the universal principally polarized abelian scheme (with level $N$
structure).

\begin{lemma}
\label{lemagxistrat}
Fix $\xi \in \Xi_{n,g}(\ff)$ with $1 \le n \le \infty$.  There exists
a reduced subscheme $\stk A_{g,\xi} \subset \stk A_g$ such that, for each algebraically
closed field $k$ of characteristic $p$,
\begin{equation}
\label{eqdefagxi}
\stk A_{g,\xi}(k) = \st{(X,\lambda) \in \stk A_g :
  (X[p^n],\lambda[p^n]) \iso \xi_{k}}.
\end{equation}
Moreover, $\stk A_{g,\xi}$ is affine, smooth and equidimensional.
\end{lemma}

\begin{proof}
In the case $n=\infty$, the existence of a scheme $\stk A_{g,\xi}$
satisfying \eqref{eqdefagxi} is \cite[Thm.\ 3.3]{oortfoliations};
there, such a scheme is called a ``central leaf.''
For
finite $n$ see, e.g., \cite[Thm.\ 1.2]{vasiulevelm}.   It has been understood for some time that
$\stk A_{g,\xi}$ is pure of codimension $\dim \aut(\xi)$ in $\stk A_g$
\cite{vasiulevelm,wedhorndimoort}.  For smoothness, we model our proof on that of
\cite[Thm.\ 3.13]{oortfoliations}, which proves the smoothness in the
case $n=\infty$.  (See also \cite[Sec.\ 4]{vasiulevelm}.)  Suppose $k$ is algebraically closed and $s,t \in
\stk A_{g,\xi}(k)$.  Choose an isomorphism $(\stk X,\lambda)_s[p^n]
\sriso (\stk X,\lambda)_t[p^n]$; it induces an isomorphism of
deformation spaces $\defo((\stk X,\lambda)_s[p^n]) \sriso
\defo((\stk X,\lambda)_t[p^n])$.  By \cite[Thm.\ 2.8]{wedhorndimoort},
$\defo((\stk X,\lambda)_s[p^\infty]) \ra
\defo((\stk X,\lambda)_s[p^n])$ is formally smooth; the Serre-Tate theorem shows that the
formal completion $\stk A_{g}^{/s}$ is isomorphic to
$\defo((\stk X,\lambda)_s[p^\infty])$.  Of course, the same statements
hold at $t$, {\em mutatis mutandis}, and taken together this shows
that $\stk A_{g,\xi}^{/s} \subset \stk A_{g}^{/t}$ and $\stk
A_{g,\xi}^{/t} \subset \stk A_g^{/t}$ are isomorphic.  Since we have
given $\stk A_{g,\xi}$ a structure of reduced subscheme, all of its
points are equisingular, thus smooth.
\end{proof}

The Newton polygon of a \btinfinity is any device which records its
geometric isogeny class; the Newton polygon of an abelian variety $X$
is that of $X[p^\infty]$.  For an admissible, symmetric Newton polygon $\nu$ of
height $2g$, let $\stk A_g^\nu$ be the reduced, locally closed
subscheme parametrizing those abelian varieties
whose Newton polygon is $\nu$.

Oort has shown that $\stk A_g^\nu$ is, up to a finite morphism, a
product, as follows.

\begin{theorem}[Oort]
\label{thoortfoliation}
Let $\nu$ be an admissible symmetric Newton polygon of height $2g$.
There exist schemes $S^\nu$ and $T^\nu$ of finite type over $\ff_p$, a
geometrically isotrivial $p$-divisible group $\calg^\nu \ra S^\nu
\cross T^\nu$ with Newton polygon $\nu$, and a finite surjective
morphism
\begin{diagram}[LaTeXeqno]
\label{diagoort}
S^\nu \cross T^\nu & \rto^{\Phi^\nu} & \stk A_g ^\nu
\end{diagram}
such that
\begin{alphabetize}
\item There is an isogeny $\calg \ra \Phi^{\nu *}\stk X[p^\infty]$.
\item For each $s \in S^\nu(\ff)$ there exists some $\xi \in
  \Xi_{g,\infty}(\ff)$ such that
\[
\Phi^\nu(s\cross T^\nu) = \stk A_{g,\xi}.
\]
\item Each $\stk A_{g,\xi}$ arises in this way.
\end{alphabetize}
\end{theorem}

\begin{proof}
This is \cite[Thm.\ 5.3]{oortfoliations}.   In fact, what Oort proves is
that, over $\ff$, there exists an arrangement as in  \eqref{diagoort}
where the domain is a product of {\em integral} schemes over $\ff$.
The present formulation follows by taking the disjoint union of
$\gal(\ff_p)$-conjugates of the schemes constructed in {\em loc. cit.}
\end{proof}

Consequently, at least if we set aside the supersingular locus,
central leaves have bounded topology.  (In general,
say that $\xi \in \Xi_{g,n}(\ff)$ is not supersingular if no component
of $\stk A_{g,\xi}$ is fully contained in the supersingular locus of
$\stk A_g$.)  For a variety $V$ over a field $K$, let $\sigma_c(V)$ be the sum of its compact Betti numbers; 
\[
\sigma_c(V) = \sum_i \dim H^i_c(V_{\bar K},\bar\rat_\ell).
\]

\begin{corollary}
\label{cortrivtop}
For each natural number $N$ there exists $C = C(g,N,p)$ such that, for
each non-supersingular $\xi \in \Xi_{g,\infty}(\ff)$, there exists a scheme $V_{\xi,N}$
and a finite and faithfully flat  morphism
\begin{diagram}
V_{\xi,N} & \rto^{\phi_{\xi,N}} & \stk A_{g,\xi}
\end{diagram}
such that $\phi^*_{\xi,N}\stk X[p^N]$ is constant and $\sigma_c(V_{\xi,N}) < C$.
\end{corollary}

\begin{proof}
All necessary information is contained in Oort's proof of Theorem
\ref{thoortfoliation}.  Subsequent improvements by Chai and Oort
\cite{chaioort11}  allow a somewhat clearer picture of the structure of
\eqref{diagoort}, and we avail ourselves of these insights to give a
streamlined, if slightly anachronistic, description of Theorem
\ref{thoortfoliation}.


Fix a non-supersingular Newton polygon $\nu$.  The uniformization
\eqref{diagoort} is achieved as follows.  Let $W = \stk A_g^\nu$; it is
smooth and irreducible.  There is an abelian scheme $\stk Y \ra
W$, and an isogeny $\stk Y \ra \stk X\rest W$ (say of degree $p^i$), such
that $\stk Y$ admits a polarization of degree $p^i$ and $\stk Y[p^\infty]$
is completely slope divisible.  In particular, the existence of $\stk Y
\ra W$ means there is a map $W \ra \stk A_{g,p^i}$, whose image turns
out to be a central leaf $C\subset \stk A_{g,p^i}^\nu$.  Fix a central
leaf $D\subset W$; there is a finite, surjective morphism $D \ra C$.
In \eqref{diagoort}, for $T^\nu$ one takes a finite, faithfully flat  $T \ra C$ which
trivializes $\stk Y[p^i]$ (Proposition \ref{proptrivial}) and then sets $T^\nu = T \cross_C D$, which is
finite and faithfully flat over $D$.  For $S^\nu$, one takes any isogeny leaf in $\stk
A_g^\nu$.

Let $Y$ be any fiber of $\stk Y \ra W$.  On one hand, $\stk Y[p^i]_{T^\nu} =
Y[p^i]\cross T^\nu$.  On the other hand, there is a finite group
scheme $\calg/S^\nu$ and an inclusion
$\calg \inject Y[p^i] \cross S^\nu$ such that $(Y\cross S^\nu)/\calg
\iso \stk X \rest{S^\nu}$ (as polarized abelian schemes).  Ultimately,
one has
\begin{diagram}
\calg_{S^\nu\cross T^\nu} & \rinject & Y[p^i] \cross S^\nu \cross
T^\nu &\iso \stk Y[p^i]_{S^\nu\cross T^\nu}.
\end{diagram}
Suppose $s \in S^\nu(\ff)$; let $\xi = \xi(s)$ be the corresponding
isomorphism class of \btn.  The universal abelian scheme over the
central leaf $\stk A_{g,\xi}$ in $W$ is given by
\[
\stk X\rest{\stk A_{g,\xi}} \iso \stk Y\rest{T^\nu}/(\calg_s \cross T^\nu).
\]
In particular, as in Proposition \ref{proptrivial}, let $T^{\nu, N+i} \ra T^\nu$ be a finite and faithfully flat morphism 
(with $T^{\nu,N+i}$ regular) which trivializes 
trivializes $\stk Y[p^{N+i}] \rest{T^\nu}$.  Then for each central
stream $\stk A_{g,\xi} \subset \cala_g^\nu$, the finite and surjective morphism
$\phi_{\xi,N}: T^{\nu,N+i}
\ra \stk A_{g,\xi}$ trivializes $\stk X[p^N]\rest{\stk A_{g,\xi}}$.  Since the source and target of $\phi_{\xi,N}$ are regular, the morphism is faithfully flat.

The corollary now follows, where for $C$ we may take
\[
C = \max\st{ \sigma_c(T^{\nu,N+i}) : \nu\text{ not
    supersingular}}.
\]
\end{proof}

In the proof of Corollary \ref{cortrivtop}, we needed the fact that a geometrically constant \btn over a regular base is fppf-trivial, and is in fact trivialized by pullback to a regular scheme:  

\begin{proposition}[Oort]
\label{proptrivial}
Let $S/K$ be a regular Noetherian scheme over a perfect field, and let $\calg \ra S$ be a geometrically fiberwise constant $p$-divisible group.   For each natural number $N$, there exists a Noetherian regular scheme $T_N/K$ and a finite and faithfully flat morphism $T_N \ra S$ such that $\calg \cross_S T_N$ is constant.
\end{proposition}

\begin{proof}
This is \cite[Thm.\ 1.3]{oortfoliations}.  In fact, Oort shows the existence of a finite, surjective morphism $T_N \ra S$ which trivializes $\calg[p^N]$, under the weaker hypothesis that $S$ be Noetherian and integral with Noetherian normalization.  Moreover, $T_N \ra S$ factors as $T_N \ra T_N' \ra S_N \ra S$, where $T_N \ra T_N'$ is \'etale; $T_N' \ra S_N$ is purely inseparable; and the normalization morphism $\til S \ra S$ factors as $\til S \ra S_N \ra S$.  To prove the formulation in Proposition \ref{proptrivial}, we may assume that $S$ is integral.  The normality of $S$  implies that $S_N \iso S$.  An analysis of \cite[Lemma 1.4]{oortfoliations} shows that one may realize  $T_N' \ra S_N$ as a suitable iterate of the Frobenius morphism $F^{(i)}: T_N' \ra T_N'^{(p^i)}$. Then $T_N \ra S$, as a composition of finite and faithfully flat morphisms, is finite and faithfully flat itself.  Finally, suppose $Q \in T_N'$.  Since $K$ is perfect, the local rings $\calo_{T_N',Q}$ and $\calo_{T_N^{(p^i)}, F^{(i)}(Q)}$ are (abstractly) isomorphic.  Therefore $T_N$, as an \'etale cover of a regular scheme, is regular.  
\end{proof}

Proposition \ref{proptrivial} shows that a geometrically fiberwise constant \btinfinity over a regular base is locally constant in the sense of \cite{garuti08}.  Let $s \in S$ be a point, with residue field $\kappa(s)$.  By \cite[Prop.\ 3]{garuti08}, in the setting of Proposition \ref{proptrivial}, for $T_n$ one may take a certain torsor under a finite group scheme over $\kappa(s)$.

Say that $\xi \in \Xi_{g,n}(\ff)$ admits a generic Newton polygon if
there is some Newton stratum $\stk A_g^{\nu(\xi)}$ such that $\stk A_g^{\nu(\xi)}
\cap \stk A_{g,\xi}$ is dense in $\stk A_{g,\xi}$.  For example, if
$\xi \in \Xi_{g,1}(\ff)$ is not supersingular, then $\stk A_{g,\xi}$ is irreducible
\cite[Thm.\ 11.5]{evdg} (see also \cite{harashitagenericnp}), and
thus $\xi$ admits a generic Newton polygon.  At the opposite extreme,
if $n \gg_g 0$, then all pqp $p$-divisible groups with $p^n$-torsion
$\xi$ are isomorphic, and in particular share the same Newton
polygon.  More generally, Lau, Nicole and Vasiu give an explicit upper
bound \cite[Thm.\ 1.2]{launicolevasiu} for a function (``the isogeny cutoff'') $b(g)$ such
that, if $\xi \in \Xi_{g,n}$ and $n \ge b(g)$, then $\stk A_{g,\xi}$
is contained in a unique Newton stratum.    For such a $\xi$, denote
its generic Newton polygon by $\nu(\xi)$.

A stratum $\stk A_{g,\xi}$ with a generic Newton polygon is irreducible:

\begin{theorem}[Chai-Oort]
Suppose $1 \le n \le \infty$ and that $\xi \in \Xi_{g,n}(\ff)$ has 
a generic, non-supersingular Newton polygon.
     Then $\stk A_{g,\xi}$ is irreducible.
\end{theorem}

\begin{proof}
If $n = \infty$ (equivalently, if  $n\gg_g0$), then this  is
\cite[Thm.\ 4.1]{chaioort11}.  For arbitrary $n$, only a very modest adaptation of Chai and
Oort's argument is necessary.  It suffices to
show that $\stk A_{g,\xi}^{\nu(\xi)} := \stk A_{g,\xi} \cap \stk
A_g^{\nu(\xi)}$ is irreducible.  Each irreducible
component of $\stk A_{g,\xi}^{\nu(\xi)}$ contains a central leaf, and in
particular contains the moduli point of a principally polarized abelian variety
isogenous to a hypersymmetric one.  Consequently, prime-to-$p$ Hecke
correspondences act transitively on the irreducible components of
$\stk A_{g,\xi}^{\nu(\xi)}$.  Since $\stk A_{g,\xi}^{\nu(\xi)}$ is
smooth (Lemma \ref{lemagxistrat}), \cite{chailadic} lets us deduce the
irreducibility of  $\stk A_{g,\xi}^\nu$.
\end{proof}

\section{The distribution of forms of group schemes}
\label{secdistribution}

In all that follows, we will assume that $\xi \in \Xi_{g,n}(\ff)$ is
not supersingular but does admit a generic Newton polygon.

\subsection{Generic behavior}

Suppose $1 \le n < m \le \infty$ and $\xi \in \Xi_{g,n}(\ff)$.  Then
there exists at least one lift of $\xi$ to $\Xi_{g,m}(\ff)$, i.e., there is
some $\til \xi \in \Xi_{g,n}(\ff)$ such that $\til \xi[p^n] = \xi$
(\cite[Prop.\ 3.2]{wedhorndimoort}, which is deduced from the
unpolarized verison \cite[Prop.\ 1.7]{illusiebt}).  Then
\[
\stk A_{g,\xi}(\ff) = \cup_{\til \xi \in \Xi_{g,m}(\ff)_\xi} \stk
A_{g,\til\xi}(\ff).
\]
Moreover, $\cup_{\til \xi \in \in \Xi_{g,m}(\ff)_\xi} \stk A_{g,\til
  \xi}$ is  a disjoint union of countably many locally closed
subschemes of $\stk A_g$, and is Zariski dense in $\stk A_{g,\xi}$.

\begin{lemma}
\label{lemdefamxi}
Fix $1 \le n < m \le \infty$.
\begin{alphabetize}
\item There exists a subgroup $A_m(\xi) \subseteq A(\xi)$, and an open
  dense subscheme $\stk A_{g,m,\xi}^\circ \subseteq \stk A_{g,\xi}$,
  such that if $s \in \stk A_{g,m,\xi}^{\circ}(\ff)$, then the image of
\begin{diagram}[LaTeXeqno]
\label{diagamxi}
\aut_{\stk X_s[p^m],\lambda_s[p^m]} & \rto & \aut_{\stk X_s[p^n],\lambda_s[p^n]} & \rto & \pi_0(\aut_{\stk X_s[p^n],\lambda_s[p^n]})= A(\xi)
\end{diagram}
is $A_m(\xi)$.
\item There is a subset $\Xi_{g,m}^{\circ}(\ff)_\xi \subseteq \Xi_{g,m}(\ff)$ such that
\[
\stk A_{g,m,\xi}^{\circ}(\ff) = \cup_{\til \xi \in \Xi_{g,m}^\circ(\ff)_\xi} \stk A_{g,\til\xi}(\ff).
\]
\end{alphabetize}
\end{lemma}

\begin{proof}
  Let $\rho:\aut_{\stk X[p^m],\lambda[p^m]} \ra
  \aut_{\stk X[p^n],\lambda[p^n]}$ be the obvious homomorphism of group
  schemes over $\stk A_{g,\xi}$, and let $K_{m,n}$ be its kernel.  Since
  $\stk A_{g,\xi}$ is Noetherian, irreducible and reduced, there is a
  finite decomposition $\stk A_{g,\xi} = \cup_{0 \le i \le r} V_i$
  of $\stk A_{g,\xi}$ as a disjoint union of locally closed subschemes
  such that $K_{m,n}\rest{V_i}$ and
  $\aut_{\stk X[p^m],\lambda[p^m]}\rest{V_i}$ are flat, and $V_0$ is
  open and dense in $\stk A_{g,\xi}$.  Note that, over each $V_i$, the
  quotient functor $H := \aut_{\stk X[p^m],\lambda[p^m]}/K_{m,n}$ is
  represented by a group scheme.  Possibly after refining the
  decomposition of $\stk A_{g,\xi}$, we assume that $H\rest{V_i}$ is
  flat, too.

Since $V_0$ is Noetherian and $\pi_0(H\rest{V_0})$ is constructible, there is an open dense subscheme $U_0 \subseteq V_0$ over which $\pi_0(H)$ has geometrically constant fibers.  Let $A_m(\xi)$ be this common group.

For each $1 \le i \le r$, let $U_i \subset V_i$ be the (possibly empty) locally closed subscheme where $\pi_0(H) \iso A_m(\xi)$.  Then $\stk A_{g,m,\xi}^{\circ} := \cup U_i$ is the sought-for open, dense subscheme.

From the definition of $\stk A_{g,m,\xi}^\circ$ it is clear that, if $s \in \stk A_{g,\xi}(\ff)$, then $s\in \stk A_{g,m,\xi}^\circ(\ff)$ if and only if \eqref{diagamxi} is surjective.  Thus, membership in  $\stk A_{g,m,\xi}^\circ$ depends only on the isomorphism class of $(X[p^m],\lambda[p^m])$, and $\stk A_{g,m,\xi}^\circ$ is a union of certain locally closed subschemes $\stk A_{g,\til\xi}$.
\end{proof}

Henceforth, let $A_*(\xi) = A_\infty(\xi)$, the group of generically
liftable discrete automorphisms of $\xi$, and let $\stk
A_{g,\xi}^\circ$ be the complement of  the supersingular locus in $\stk A_{g,\infty,\xi}^\circ$.

\subsection{Distribution of forms}

In this section, we fix a $\xi \in \Xi_{g,n}(\ff)$ which admits a
generic, non-supersingular Newton polygon.   Recall that
$\ff_{q(\xi)}$ is the minimal field of definition of a split
representative for $\xi$.  

\begin{lemma}
\label{lemmostacirc}
\begin{equation}
\label{eqastarxi}
\lim_{\ff_q/\ff_{q(\xi)}} \frac{\#\st{(X,\lambda) \in \stk
    A_{g,\xi}(\ff_q) : (X,\lambda)[p^n] \iso \xi^{\alpha,\ff_q}\text{
      for some }\alpha \in A_*(\xi)}}{\#\stk A_{g,\xi}(\ff_q)} = 1.
\end{equation}
\end{lemma}

In \eqref{eqastarxi} and similar constructions, the limit is over
ever-larger extensions of $\ff_{q(\xi)}$.

\begin{proof}
If the moduli point of $(X,\lambda)$ is contained in $\stk
A_{g,\xi}^\circ(\ff_q)$ (Lemma \ref{lemdefamxi}), then the isomorphism
class of $(X,\lambda)[p^n]$ is represented by a suitable $\alpha$.
\end{proof}

Unfortunately, except in certain special cases (such as the {\em
  minimal} \btn's discussed below), we have little 
information on the topology of the complement $\stk A_{g,\xi} \setcomp \stk
A_{g,\xi}^\circ$, and thus no explicit control over the rate of
convergence in \eqref{eqastarxi}.  Thus, at present, in order to make
statements with uniform bounds such as Theorem \ref{thmainacirc}, we
need to condition on membership in $\stk A_{g,\xi}^\circ(\ff_q)$.

\begin{lemma}
\label{lemboundgroup}
For each natural number $n$, there exists $B= B(g,n,p)$ such that the
following holds.  Let $\til \xi \in \Xi_{g,\infty}(\ff)$, and let
$\xi = \til \xi[p^n]$.  Then
\[
\abs{A(\til \xi, \xi)} < B.
\]
\end{lemma}

\begin{proof}
Let $(G,\lambda)$ be a pqp \btinfinity representing $\xi$; it suffices
to uniformly bound the size of the (discrete, abstract) group
$(\End(G,\lambda)/p^n)\units$, and thus a uniform bound on
$\abs{(\End(G)/p^n)\units}$ suffices.  Let $m(c/d)$ be the
multiplicity of the rational number $c/d$ in the Newton polygon of
$G$.  Then $\End(G)$ is an order in $\oplus \mat_{m(c/d)}(D_{c/d})$,
where $D_{c/d}$ is the $\rat_p$-algebra with invariant $c/d$, and thus
$\abs{(\End(G)/p^n)\units}$ admits an upper bound depending only on
$p$, $n$ and the Newton polygon of $G$.  Since, for fixed $g$, there
are only finitely many Newton polygons which arise, the result
follows.
\end{proof}

\begin{lemma}
\label{lembasicestimate}
Let $\til \xi \in \Xi_{g,\infty}(\ff)$ be
  non-supersingular, and let $\xi = \til \xi[p^n]$.  For each $\alpha
  \in A(\til \xi, \xi)$ and each finite extension $\ff_q$ of
  $\ff_{q(\xi)}$,
\[
\frac{
\#\st{ (X,\lambda)\in \stk A_{g,\til\xi}(\ff_q): (X,\lambda)[p^n] \iso \xi^{\alpha,\ff_q}}}{\# \stk A_{g,\til\xi}(\ff_q)}
 = \oneover{\# Z_{A(\til \xi, \xi)}(\alpha)} + \bigoh_{g,n,p}(1/\sqrt q).
 \]
\end{lemma}

In other words, there exists $D = D(g,n,p)$ such that 
\[
\abs{
\frac{
\#\st{ (X,\lambda)\in \stk A_{g,\til\xi}(\ff_q): (X,\lambda)[p^n] \iso \xi^{\alpha,\ff_q}}}{\# \stk A_{g,\til\xi}(\ff_q)}
- \oneover{\# Z_{A(\til \xi, \xi)}(\alpha)}
}
< \frac{D}{\sqrt q}.
\]

\begin{proof}
Let $s \in \stk A_{g,\til\xi}$ be a geometric point.  By \cite[Thm.\ 5.6]{chaioort11}, the na\"ive monodromy representation
\begin{diagram}
\pi_1(\stk A_{g,\til\xi},s) & \rto & \aut(\stk X_s[p^\infty],\lambda_s[p^\infty])) \iso \aut(\til\xi)
\end{diagram}
is surjective.  Thus, the image of the composite homomorphism
\begin{diagram}
\rho_{\til\xi, \xi}:\pi_1(\stk A_{g,\til\xi},s) & \rto & \aut(\til\xi) & \rto & \aut(\xi) & \rto & \pi_0(\aut(\xi))
\end{diagram}
is all of $A(\til \xi, \xi)$.   Moreover, the finite and faithfully
flat morphism $V_{\til \xi, n} \ra \stk A_{g, \til\xi}$ of Corollary
\ref{cortrivtop} trivializes $(\stk X[p^n],\lambda[p^n])$, and thus the
representation $\rho_{\til \xi, \xi}$; and there is a bound $C= C(g,n,p)$ for
the sum $\sigma_c(V_{\til\xi,n})$ of the $\ell$-adic compact Betti numbers of $V_{\til \xi, n}$. The image $A(\til
\xi,\xi)$ of 
$\rho_{\til\xi,\xi}$ has size at most $B = B(g,n,p)$ (Lemma \ref{lemboundgroup}).
The claim now follows from the
explicit Chebotarev-type theorem of Katz and Sarnak \cite[Thm.\ 9.7.13]{katzsarnak}.  By \cite[Thm. 9.2.6.(4)]{katzsarnak}, for the constant they call $C(\stk A_{g,\til\xi}/\spec \ff_{q(\xi)})$, one may take any upper bound for the sum of the compact Betti numbers of an \'etale Galois cover of $\stk A_{g,\til\xi}$ which trivializes $\rho_{\til\xi,\xi}$; our $C$ provides such a bound.
Then Frobenius elements are asymptotically equidistributed in $A(\til
\xi,\xi)$, with error term bounded by $2 B \cdot C / \sqrt q$.  In fact, the Katz-Sarnak estimate only applies when $q$ is at least four times the square of $\sigma_c(\stk A_{g,\til\xi})$; but since this may be bounded uniformly in $g$, $n$ and $p$, the claimed inequality still holds, possibly at the cost of replacing $2B\cdot C$ with a larger constant $D$.
\end{proof}

\begin{theorem}
\label{thmainacirc}
Suppose $\xi \in \Xi_{g,n}(\ff)$ hs a generic Newton polygon which is not supersingular. Then for each $\alpha \in A_*(\xi)$ and each finite extension $\ff_q$ of $\ff_{q(\xi)}$,
\[
\frac{
\#\left\{ (X,\lambda)\in \stk A^\circ_{g,\xi}(\ff_q): (X,\lambda)[p^n] \iso \xi^{\alpha,\ff_q}\right\}}{\# \stk A^\circ_{g,\xi}(\ff_q)}
=
\oneover{\# Z_{A_*(\xi)}(\alpha)}
+ \bigoh_{g,n,p}(1/\sqrt q).
\]
\end{theorem}

\begin{proof}
Decompose $\stk A^\circ_{g,\xi}(\ff_q)$ as a disjoint union
\begin{align*}
\stk A_{g,\xi}^{\circ}(\ff_q) &= \cup_{\til \xi \in \Xi_{g,\infty}^\circ(\ff)_\xi} \stk A_{g,\til\xi}(\ff_q)\\
\intertext{and then apply Lemma \ref{lembasicestimate}:}
\# \st{ (X,\lambda)\in \stk A^\circ_{g,\xi}(\ff_q): (X,\lambda)[p^n] \iso \xi^{\alpha,\ff_q}} &= \sum_{\til \xi \in \Xi_{g,\infty}^\circ(\ff)_\xi} 
\#\st{ (X,\lambda)\in \stk A_{g,\til\xi}(\ff_q): (X,\lambda)[p^n] \iso \xi^{\alpha,\ff_q}}\\
&= 
\sum_{\til \xi \in \Xi_{g,\infty}^\circ(\ff)_\xi} 
\# \stk A_{g,\til\xi}(\ff_q) \cdot \left( \oneover{\#Z_{A_*(\xi)}(\alpha)}+ \bigoh_{g,n,p}(1/\sqrt q)\right) \\
&= \# \stk A_{g,\xi}^\circ(\ff_q)\cdot \left( \oneover{\#Z_{A_*(\xi)}(\alpha)}+ \bigoh_{g,n,p}(1/\sqrt q)\right).
\end{align*}
\end{proof}

\section{Examples and complements}
\label{secexample}

We close with some explicit examples of the groups $A(\xi)$ defined in Section \ref{secpointwise}, and a brief discussion of related work in the literature concerning the distribution of $p$-power torsion group schemes in abelian varieties.

\subsection{Simple minimal $p$-divisible groups}
\label{subsecexminimal}

Fix relatively prime natural numbers $c$ and $d$; let $h = c+d$ and
let $\mu = d/(c+d)$; and work over any perfect field $K$
containing $\ff_{p^h}$. Then there is a minimal $p$-divisible group
$H_{c,d}/K$ of height $h$ and slope $\mu$.  It is characterized by
the fact that its endomorphism ring is maximal among all such
$p$-divisible groups; $\End(H_{c,d}) \iso \calo_\mu$, the
maximal order in the $\rat_p$-division algebra of Brauer invariant
$\mu$ \cite{oortminimal}. An important fact about minimal $p$-divisible groups is that
they are characterized by their $p$-torsion; if $G$ is any
$p$-divisible group with $G[p] \iso H_{c,d}[p]$, then $G \iso H_{c,d}$.

On one hand, $\aut(H_{c,d})$ is the pro-\'etale group
$\calo_\mu\units$.  On the other hand, a direct calculation shows that $\pi_0(\aut(H_{c,d}[p])) \iso \ff_{p^h}\units$.  

Rather than prove these statements in complete generality, we work out the
special, but representative, case where $(c,d) = (2,1)$.  Then
$M := \dieu_*(H_{2,1})$ admits a $W(K)$-basis $\st{x_1,x_2,x_3}$ on which
the action of Frobenius and Verschiebung are given by
\begin{align*}
F x_1 &= x_2 & Vx_1 &= x_3 \\
F x_2 &= x_3 & Vx_2 &= px_1 \\
Fx_3 &= px_1 & Vx_3 &= px_2
\end{align*}
Consider an automorphism of $\alpha$ of $M$ as $W(K)$-module.  Then
$\alpha$ is an automorphism of $M$ as Dieudonn\'e module if and only
if
\begin{align*}
A \cdot [F] &= [F] \cdot A^\sigma \\
A^\sigma \cdot [V] &= [V] \cdot A
\end{align*}
where $A= (a_{ij})$ is the matrix of $\alpha$ in the chosen basis, and $[F]$ and
$[V]$ are similarly given by
\begin{align*}
\null [F] &= \begin{pmatrix}
0 & 0 & p \\
1 & 0 & 0 \\
0 & 1 & 0
\end{pmatrix} \\
\null [V] &= \begin{pmatrix}
0 & p & 0 \\ 
0 & 0 & p \\
1 & 0 & 0
\end{pmatrix}.
\end{align*}
Direct calculation shows that $A$ must be of the form
\[
A =
\begin{pmatrix}
a_{11} & p a_{31}^\sigma & p a_{32}^\sigma\\
a_{32}^{\sigma^2} & a_{11}^\sigma & a_{31}^{\sigma^2} \\
a_{31} & a_{32} & a_{11}^{\sigma^2}
\end{pmatrix}
\]
where $a_{11}^{\sigma^3} = a_{11}$, i.e., $a_{11} \in W(\ff_{p^3})$,
and $a_{31}$ and $a_{32}$ similarly lie in $W(\ff_{p^3})$.

We now consider the matrix $\bar A$ of an automorphism $\bar A$ of
$M/pM = \dieu_*(H_{1,2}[p])$.  The general form of such an
automorphism is
\[
\bar A =
\begin{pmatrix}
\bar a_{11} & 0 & 0 \\
\bar a_{2,1} & \bar a_{1,1}^\sigma & 0 \\
\bar a_{3,2} &\bar a_{2,1}^\sigma & \bar a_{1,1}^{\sigma^2}
\end{pmatrix}
\]
where $\bar a_{1,1}^{\sigma^3} = \bar a_{1,1}$.  Consequently,
$\aut(H_{1,2}[p])$ sits in an exact sequence
\begin{diagram}
1 & \rto & \aut(H_{1,2}[p])^0 & \rto & \aut(H_{1,2}[p]) & \rto &
\ff_{p^3}\units & \rto & 1,
\end{diagram}
where the connected component of identity is a two-dimensional unipotent group. 

For general $(c,d)$, one finds that 
\begin{align*}
\pi_0(\aut(H_{c,d}[p])) & \iso \ff_{p^h}\units \\
\intertext{while }
\pi_0(\aut(H_{c,d}))/p & \iso (\calo_\lambda/p)\units \iso \gl_h(\ff_{p^h}).
\end{align*}
Furthermore, although
\begin{diagram}
\aut(H_{c,d}) & \rto & \aut(H_{c,d}[p])
\end{diagram}
is far from surjective -- the source is pro-\'etale, while the target has positive dimension -- the induced map
\begin{diagram}
\aut(H_{c,d}) & \rto & \aut(H_{c,d}[p]) & \rto & \pi_0(\aut(H_{c,d}[p]))
\end{diagram}
{\em is} surjective.

\subsection{Simple pqp minimal $p$-divisible groups}
\label{subsecexpqpminimal}

Now suppose that $(c,d) \not = (1,1)$, and let $G_{\mu,1-\mu} = H_{c,d} \oplus H_{d,c}$.  There is a canonical principal $\lambda_{\mu,1-\mu}$ on $G_{\mu,1-\mu}$
quasipolarization $\lambda_{\mu}$ on $H_{c,d} \oplus H_{d,c}$.  One has
\[
\aut(G_{\mu,1-\mu},\lambda_{\mu,1-\mu}) \iso \aut(H_{c,d})
\]
and the analogous statement holds for the $p$-torsion group schemes, as well.  Let $\til\xi \in \Xi_{2h,\infty}(\ff)$ represent the geometric isomorphism class of $(G_{\mu,1-\mu},\lambda_{\mu,1-\mu})$, and let $\xi \in \Xi_{2h,1}(\ff)$ be its $p$-torsion. Since $G_{\mu,1-\mu}$ is minimal, we have
\[
\stk A_{2h,\xi} = \stk A_{2h,\til\xi} = \stk A^\circ_{2h,\xi}.
\]
Moreover, 
\[
A(\til \xi, \xi) = A_*(\xi) \iso \ff_{p^h}\units.
\]

\subsection{Supersingular elliptic curves}
\label{subsecexss}

In contrast to Section \ref{subsecexpqpminimal}, we now consider the
case where $c=d=1$, and restrict our discussion to fields which contain $\ff_{p^2}$. Over an algebraically closed field $k$ of characteristic $p$, $H_{1,1,k}$ is isomorphic to the $p$-divisible group
of any supersingular elliptic curve over $k$, and $H_{1,1}[p]_k$ to its $p$-torsion.  Let $\lambda$ be a principal quasipolarization on $H_{1,1}$.  If $a \in \ff_{p^2}\units \subset \aut(H_{1,1}[p])$, then $a$ scales $\lambda$ by $a a^\sigma$.  Consequently, 
\begin{align*}
\pi_0(\aut(H_{1,1}[p])) & \iso \ff_{p^2}\units \\
\pi_0(\aut(H_{1,1}[p],\lambda[p])) & \iso \st{ a \in \ff_{p^2}\units: a a^\sigma = 1}.
\end{align*}
In particular, we find that $H_{1,1}[p]$ admits $p^2-1$ distinct
$\ff_q$-forms; as a principally quasipolarized \btone, it still admits
$p+1$ distinct $\ff_q$-forms.  (Unfortunately, this is not in accordance with \cite[Sec.\
6]{liedtke11}; it seems that, in the cited work, noncanonical
identifications of $H_{1,1}[F]$ and its cokernel with $\aalpha_p$ led to an incorrect conclusion.)

In fact, following a suggestion of Oort, we can be explicit about these different forms.  We work over a perfect field $K$ containing $\ff_{p^2}$.  Define a map
\begin{diagram}
K\units & \rto^\tau & K\units \\
a  & \rmto & a^\sigma/a^{\sigma\inv}.
\end{diagram}
We will exhibit a bijection between forms of $H_{1,1}[p]$ and the quotient $K\units/\tau(K\units)$.  (Note that $\ker \tau = \ff_{p^2}\units$ and thus, if $K$ is a finite field, $K\units/\tau(K\units)\iso \ff_{p^2}\units$.)

  Let $\bar M = \dieu_*(H_{1,1}[p])$; since $\dim
  \hom(\aalpha_p, H_{1,1}[p]) = 1$, $\bar M$ is generated as a $K[F,V]$-module by a single element.  (In the coordinates introduced in Section \ref{subsecexminimal}, the element $x_1$ suffices.)
Moreover, any twist $\bar N$ of $\bar M$ has this property, too.

For such an $\bar N$ and generator $x$, let $y = Fx$ and define $\lambda_{\bar N,x} \in K\units$ by $Vx = \lambda_{\bar N,x}y$. Any other generator for $\bar N$
is of the form $bx+cy$, where $b\in K\units$ and $c \in K$.  For this generator, we find that 
\begin{align*}
F(bx+cy) &= b^\sigma Fx = b^\sigma y \\
V(bx+cy) &= b^{\sigma\inv}\lambda_{\bar N,x}y\\
&= b^{\sigma\inv}(b^\sigma)\inv \lambda_{\bar N,x} F(bx+cy).
\end{align*}
Therefore, $\bar N$ canonically determines a class $\lambda_{\bar N} \in K\units/\tau(K\units)$.

Conversely, if $a\in K\units$, then we can define a form $\bar M_a$ of $\bar M$ as follows.  As a $K$-vector space, it is generated by $x$ and $y$; and the actions of Verschiebung and Frobenius are respectively (since we use covariant Dieudonn\'e theory) given by 
\begin{align*}
Fx &= y \\
Vx &= ay.
\end{align*}
Then $\lambda_{\bar M_a, x} = a$, and $\lambda_{\bar M_a}$ is the class of $a$ in $K\units/\tau(K\units)$.

\subsection{The ordinary locus}
\label{subsecexord}

The techniques developed in the present paper recover known results on
the structure of the physical $p$-torsion of the ordinary locus.
Indeed, fix a natural number $n$ and a dimension $g$; then the
$p^n$-torsion group scheme of an ordinary abelian variety of dimension
$g$ is geometrically isomorphic to $(\mmu_{p^n}\oplus
\integ/p^n)^{\oplus 2g}$.  There is a unique principal
quasipolarization on this group scheme, and we let $\xi_{g,\ord}$ be
the corresponding pqp \btn.  Then $\aut(\xi_{g,\ord}) \iso
\aut((\integ/p^n)^{\oplus 2g}) \iso \gl_g(\integ/p^n)$.  Any such
automorphism lifts to arbitrarily high level, and $A_*(\xi_{g,\ord})
\iso \gl_g(\integ/p^n)$.   We note, however, that there is a direct
proof (Lemma \ref{lemnaivemono}) of Theorem \ref{thmainacirc} in this
case.

\subsection{Physical $p^n$-torsion}
\label{subsecexphys}

For applications (e.g., \cite{castrycketal12}), it is useful to be
able to quantify the distribution of physical $p$-power torsion.
Although Theorem \ref{thmainacirc} can be adapted to such
considerations, a direct argument is perhaps more transparent.

Let $X \ra S/\ff_{q_0}$ be an abelian scheme of dimension $g$ over a
connected base.  Suppose that $X$ has constant $p$-rank $f \ge 1$, in
the sense that, for each point $s\in S$, $X_s[p](\bar{\kappa(s)}) \iso
(\integ/p)^{\oplus f}$.  Then, after a choice of geometric point $\bar
s$ of $S$, for each natural number $n$ there is a monodromy representation
\begin{diagram}
\pi_1(S,\bar s) & \rto^{\rho_{X/S,n}} & \aut(X[p^n]^\et) \iso
\gl_f(\integ/p^n).
\end{diagram}
If $\ff_q/\ff_{q_0}$ is any extension and if $t \in S(\ff_q)$, then
there is an element $\sigma_{t/\ff_q} =   \rho_{X/S,n}(\frob_t) \in \gl_f(\integ/p^n)$,
well-defined up to conjugacy.  The $q$-power map $a\mapsto a^q$ acts
on $X_t[p^n](\bar\ff_q)$ as $\sigma_{t/\ff_q}$, and this action determines the isomorphism class of $X_t[p^n]^\et$.

\begin{lemma}
\label{lemnaivemono}
With notation as above, suppose that $\rho_{X/S,n}$ is surjective.
For each $\alpha \in \gl_f(\integ/p^n)$ and each $\ff_q/\ff_{q_0}$, 
\begin{equation}
\label{eqnaivemono}
\frac{\#\st{ t\in S(\ff_q) : \sigma_{t/\ff_q} \twiddle \alpha
  }}{\#S(\ff_q) } = 
  \oneover{\#Z_{\gl_f(\integ/p^n)}(\alpha)} +
\bigoh_{X/S}(1/\sqrt q).
\end{equation}
\end{lemma}

\begin{proof}
This is presumably well-known.  However, since there seems to be some
uncertainty in the literature (e.g., the discussion surrounding
Principle 3 in \cite{castrycketal12}), we provide a sketch.  The
\'etale sheaf
\[
\cali := \Isom(X[p^n]^\et, (\integ/p^n)^{\oplus f})
\]
on $S$ is represented by a scheme.  The scheme $\cali$ is visibly an
\'etale cover of $S$, with geometric fiber $\gl_f(\integ/p^n)$; the
hypothesis that $\rho_{X/S,n}$ is surjective is equivalent to the
irreducibility of $\cali$.  The claimed equidistribution now follows
from the Chebotarev theorem for schemes \cite[Thm.\ 7]{serrezetaandl}.
\end{proof}

In many applications, there exists a (partial) compactification $X \ra
T$ of the family, and the \'etale cover $\cali \ra S$ extends to a
wildly ramified cover of $T$.  Although some of the explicit estimates of
\cite{katzsarnak} need not apply -- see, for example, the discussion
in \cite[Rem.\ 4.8]{kowalskisieve} -- there is always a suitable
Chebotarev-type theorem, with error term $\bigoh(1/\sqrt q)$.  In the
context of Lemma \ref{lemnaivemono},  the apparent counterexamples in
\cite{kowalskisieve} stem from a na\"ive attempt to estimate the Betti
numbers of $\cali$ in terms of the Betti numbers of $S$ and the degree
of $\gl_f(\integ/p^n)$.  The sum of the compact Betti numbers of
$\cali$, whatever it might be, is the input to an explicit error term in \eqref{eqnaivemono}.

Let $\stk A_g^f \subset \stk A_g$ be the moduli space of principally
polarized abelian varieties of $p$-rank $f$.  The surjectivity of each
$\rho_{X/\stk A_g^g,n}$ is well-known
\cite{ekedahlmono,faltingschai}.  For $1 \le f \le g-1$, the
surjectivity of $\rho_{\stk X/\stk A_g^f,n}$ follows from the surjectivity of $\rho_{\stk X/\stk A_f^f, n}$, the existence of the natural inclusion $\stk A_f^f \cross \stk A_{g-f}^0 \inject \stk A_g^f$, 
and the
irreducibility of $\stk A_g^f$.

\subsection{The torsion of Jacobians}

If $C/\ff_q$ is a smooth, projective curve, then the class group of its function field is isomorphic to $\jac(C)(\ff_q)$.  Consequently, results about the distribution of torsion on (hyperelliptic) {\em Jacobians} have implications for Cohen-Lenstra type results for function fields.  For prime-to-$p$ torsion see, e.g., \cite{achtercag} for fixed $g$, large $q$ results; and \cite{ellenbergvenkateshwesterland} for the much more difficult case of fixed $q$ and large $g$.

For physical $p$-power torsion, and thus for elements of order $p$ in
the class groups of (hyperelliptic) function fields, all the results
surveyed in Sections \ref{subsecexord} and \ref{subsecexphys} have
immediate parallels for $\stk M_g^f$, the moduli space of
hyperelliptic curves of genus $g$ and $p$-rank $f$, and for the
hyperelliptic sublocus $\stk H_g^f \subset \stk M_g^f$.  Indeed, for
each irreducible component $S$ of $\stk M_g^f$ or $\stk H_g^f$, one
knows that the monodromy representation $\pi_1(S) \ra
\gl_f(\integ/p^n)$ is surjective
\cite{achterpriesprank,achterprieshyper}.  Thus, Lemma
\ref{lemnaivemono} computes the distribution of physical $p$-torsion
for (hyperelliptic) Jacobians.

However, it is not clear how to apply the methods of the present paper
to strata $\stk M_{g,\xi}$ or $\stk H_{g,\xi}$.  Indeed, it is not
even known which Newton strata in $\stk M_g$ are nonempty -- see
\cite{achterpriesricam} for a survey of what little is known -- let
alone finer strata $\stk M_{g,\xi}$, to say nothing of their $p$-adic
monodromy.  Thus, a version of Theorem \ref{thmainacirc} for 
hyperelliptic Jacobians remains a distant goal.

\bibliographystyle{habbrv} 
\bibliography{jda}

\def\cprime{$'$}
\begin{thebibliography}{10}

\bibitem{achtercag}
J.~D. Achter.
\newblock Results of {C}ohen-{L}enstra type for quadratic function fields.
\newblock In {\em Computational arithmetic geometry}, volume 463 of {\em
  Contemp. Math.}, pages 1--7. Amer. Math. Soc., Providence, RI, 2008.

\bibitem{achterpriesprank}
J.~D. Achter and R.~Pries.
\newblock Monodromy of the {$p$}-rank strata of the moduli space of curves.
\newblock {\em Int. Math. Res. Not. IMRN}, (15):Art. ID rnn053, 25, 2008.

\bibitem{achterprieshyper}
J.~D. Achter and R.~Pries.
\newblock The {$p$}-rank strata of the moduli space of hyperelliptic curves.
\newblock {\em Adv. Math.}, 227(5):1846--1872, 2011.

\bibitem{achterpriesricam}
J.~D. Achter and R.~Pries.
\newblock Generic {N}ewton polygons for curves of given $p$-rank.
\newblock In {\em Algebraic Curves and Finite Fields: Cryptography and Other
  Applications}, volume~16 of {\em Radon Series on Computational and Applied
  Mathematics}, pages 1--22. de Gruyter, Boston and Berlin, 2014.

\bibitem{caisetal13}
B.~Cais, J.~S. Ellenberg, and D.~Zureick-Brown.
\newblock Random {D}ieudonn\'e modules, random {$p$}-divisible groups, and
  random curves over finite fields.
\newblock {\em J. Inst. Math. Jussieu}, 12(3):651--676, 2013.

\bibitem{castrycketal12}
W.~Castryck, A.~Folsom, H.~Hubrechts, and A.~V. Sutherland.
\newblock The probability that the number of points on the {J}acobian of a
  genus 2 curve is prime.
\newblock {\em Proc. Lond. Math. Soc. (3)}, 104(6):1235--1270, 2012.

\bibitem{chailadic}
C.-L. Chai.
\newblock Monodromy of {H}ecke-invariant subvarieties.
\newblock {\em Pure Appl. Math. Q.}, 1(2):291--303, 2005.

\bibitem{chaioort11}
C.-L. Chai and F.~Oort.
\newblock Monodromy and irreducibility of leaves.
\newblock {\em Ann. of Math. (2)}, 173(3):1359--1396, 2011.

\bibitem{cohenlenstra}
H.~Cohen and H.~W. Lenstra, Jr.
\newblock Heuristics on class groups of number fields.
\newblock In {\em Number theory, Noordwijkerhout 1983 (Noordwijkerhout, 1983)},
  volume 1068 of {\em Lecture Notes in Math.}, pages 33--62. Springer, Berlin,
  1984.

\bibitem{ekedahlmono}
T.~Ekedahl.
\newblock The action of monodromy on torsion points of {J}acobians.
\newblock In {\em Arithmetic algebraic geometry (Texel, 1989)}, pages 41--49.
  Birkh\"auser Boston, Boston, MA, 1991.

\bibitem{evdg}
T.~Ekedahl and G.~van~der Geer.
\newblock Cycle classes of the {E}-{O} stratification on the moduli of abelian
  varieties.
\newblock In {\em Algebra, arithmetic, and geometry: in honor of {Y}u. {I}.
  {M}anin. {V}ol. {I}}, volume 269 of {\em Progr. Math.}, pages 567--636.
  Birkh\"auser Boston Inc., Boston, MA, 2009.

\bibitem{ellenbergvenkateshwesterland}
J.~Ellenberg, A.~Venkatesh, and C.~Westerland.
\newblock Homological stability for {H}urwitz spaces and the {C}ohen-{L}enstra
  conjecture over function fields, December 2009, arXiv:0912.0325.
\newblock arXiv:0912.0325.

\bibitem{faltingschai}
G.~Faltings and C.-L. Chai.
\newblock {\em Degeneration of abelian varieties}.
\newblock Springer-Verlag, Berlin, 1990.
\newblock With an appendix by David Mumford.

\bibitem{gabbervasiu}
O.~Gabber and A.~Vasiu.
\newblock Dimensions of group schemes of automorphisms of truncated
  {B}arsotti-{T}ate groups.
\newblock {\em Int. Math. Res. Not. IMRN}, (18):4285--4333, 2013.

\bibitem{garuti08}
M.~A. Garuti.
\newblock Barsotti-{T}ate groups and {$p$}-adic representations of the
  fundamental group scheme.
\newblock {\em Math. Ann.}, 341(3):603--622, 2008.

\bibitem{harashitagenericnp}
S.~Harashita.
\newblock Generic {N}ewton polygons of {E}kedahl-{O}ort strata: {O}ort's
  conjecture.
\newblock {\em Ann. Inst. Fourier (Grenoble)}, 60(5):1787--1830, 2010.

\bibitem{illusiebt}
L.~Illusie.
\newblock D\'eformations de groupes de {B}arsotti-{T}ate (d'apr\`es {A}.
  {G}rothendieck).
\newblock {\em Ast\'erisque}, (127):151--198, 1985.
\newblock Seminar on arithmetic bundles: the Mordell conjecture (Paris,
  1983/84).

\bibitem{katzsarnak}
N.~M. Katz and P.~Sarnak.
\newblock {\em Random matrices, {F}robenius eigenvalues, and monodromy}.
\newblock American Mathematical Society, Providence, RI, 1999.

\bibitem{kowalskisieve}
E.~Kowalski.
\newblock The large sieve, monodromy and zeta functions of curves.
\newblock {\em J. Reine Angew. Math.}, 601:29--69, 2006.

\bibitem{launicolevasiu}
E.~Lau, M.-H. Nicole, and A.~Vasiu.
\newblock Stratifications of {N}ewton polygon strata and {T}raverso's
  conjectures for {$p$}-divisible groups.
\newblock {\em Ann. of Math. (2)}, 178(3):789--834, 2013.

\bibitem{liedtke11}
C.~Liedtke.
\newblock The {$p$}-torsion subgroup scheme of an elliptic curve.
\newblock {\em J. Number Theory}, 131(11):2064--2077, 2011.

\bibitem{moonengsas}
B.~Moonen.
\newblock Group schemes with additional structures and {W}eyl group cosets.
\newblock In {\em Moduli of abelian varieties (Texel Island, 1999)}, pages
  255--298. Birkh\"auser, Basel, 2001.

\bibitem{oortfoliations}
F.~Oort.
\newblock Foliations in moduli spaces of abelian varieties.
\newblock {\em J. Amer. Math. Soc.}, 17(2):267--296 (electronic), 2004.

\bibitem{oortminimal}
F.~Oort.
\newblock Minimal {$p$}-divisible groups.
\newblock {\em Ann. of Math. (2)}, 161(2):1021--1036, 2005.

\bibitem{serrezetaandl}
J.-P. Serre.
\newblock Zeta and {$L$} functions.
\newblock In {\em Arithmetical {A}lgebraic {G}eometry ({P}roc. {C}onf. {P}urdue
  {U}niv., 1963)}, pages 82--92. Harper \& Row, New York, 1965.

\bibitem{serregalcoh}
J.-P. Serre.
\newblock {\em Galois cohomology}.
\newblock Springer Monographs in Mathematics. Springer-Verlag, Berlin, english
  edition, 2002.
\newblock Translated from the French by Patrick Ion and revised by the author.

\bibitem{vasiulevelm}
A.~Vasiu.
\newblock Level {$m$} stratifications of versal deformations of {$p$}-divisible
  groups.
\newblock {\em J. Algebraic Geom.}, 17(4):599--641, 2008.

\bibitem{wedhorndimoort}
T.~Wedhorn.
\newblock The dimension of {O}ort strata of {S}himura varieties of
  {P}{E}{L}-type.
\newblock In {\em Moduli of abelian varieties (Texel Island, 1999)}, pages
  441--471. Birkh\"auser, Basel, 2001.

\end{thebibliography}

\end{document}